\documentclass[12pt,a4paper]{amsart}
\usepackage[totalwidth=15.75cm,totalheight=22.275cm]{geometry}

\usepackage{amsmath,amsfonts,amsthm,amssymb,verbatim,enumerate,graphicx,mathtools,url}
\usepackage{color}
\usepackage[flushmargin]{footmisc}

\usepackage{subfig}
\newtheorem{theorem}{Theorem}[section]
\newtheorem{lemma}[theorem]{Lemma}
\newtheorem*{conjecture}{Conjecture}

\newtheorem{corollary}[theorem]{Corollary}
\newtheorem{definition}{Definition}

\newtheorem*{remark}{Remark}
\newtheorem*{remarks}{Remarks}
\numberwithin{equation}{section}
\def \Ihat  {\widehat{I}}
\def \BOhat {\widehat{BO}}
\def \BUhat {\widehat{BU}}
\def \B     {\mathcal{B}}
\def \spw {spider's web}

\newcommand{\tef}{transcendental entire function}
\newcommand\qfor{\quad\text{for }}
\newcommand \C{\mathbb{C}}
\newcommand \N{\mathbb{N}}

\makeatletter
\def\blfootnote{\xdef\@thefnmark{}\@footnotetext}
\makeatother
\begin{document}
%
%
%
%
\title[Dynamical sets whose union with infinity is connected]{Dynamical sets whose union with infinity is connected}
\author{David J. Sixsmith}
\address{Dept. of Mathematical Sciences \\
	 University of Liverpool \\
   Liverpool L69 7ZL\\
   UK \\ ORCiD: 0000-0002-3543-6969}
\email{djs@liverpool.ac.uk}
%
%
%
%
\begin{abstract}
Suppose that $f$ is a {\tef}. In 2014, Rippon and Stallard showed that the union of the escaping set with infinity is always connected. In this paper we consider the related question of whether the union with infinity of the bounded orbit set, or the bungee set, can also be connected. We give sufficient conditions for these sets to be connected, and an example a {\tef} for which all three sets are simultaneously connected. This function lies, in fact, in the Speiser class.

It is known that for many {\tef}s the escaping set has a topological structure known as a spider's web. We use our results to give a large class of functions in the Eremenko-Lyubich class for which the escaping set is not a {\spw}. Finally we give a novel topological criterion for certain sets to be a {\spw}.
\end{abstract}
\maketitle
%
%
%
%
\section{Introduction}
Let $f$ be an entire function. When studying complex dynamics it is usual to partition the complex plane into two sets; the \emph{Julia set} $J(f)$, which contains those points in a neighbourhood of which the iterates of $f$ are chaotic, and its complement the \emph{Fatou set} $F(f) = \C \setminus J(f)$. The Fatou set is open, and its connected components are called \emph{Fatou components}. For more information on complex dynamics, including precise definitions  and properties of these sets, we refer to \cite{MR1216719}.

Recently, several authors have worked with an alternative partition. This divides the plane into three sets determined by the nature of the orbits of points; the \emph{orbit} of a point $z\in\C$ is the sequence $(f^n(z))_{n\geq 0}$ of its iterates under $f$. We define these sets as follows. Firstly, the \emph{escaping} set is given by
\[
I(f) := \{ z \in \C : f^n(z) \rightarrow \infty \text{ as } n \rightarrow \infty \}.
\]
The escaping set for a general \tef\ $f$ was first studied by Eremenko \cite{MR1102727}. He showed that $ I(f) \cap J(f) \neq \emptyset $, and that all components of $ \overline{I(f)} $ are unbounded. He also conjectured that the same is true of all components of~$ I(f) $. This conjecture, which is still open, has since been the focus of much research in complex dynamics.

Secondly, the \emph{bounded orbit} set is defined by
\[
BO(f) := \{ z \in \C : \text{there exists } K > 0 \text{ such that } |f^n(z)| < K, \text{ for } n \geq 0 \}.
\]
When $f$ is a polynomial, $BO(f)$ is known as the \emph{filled Julia set}, and has been much studied. The set $BO(f)$ for a {\tef} $f$ was studied in \cite{MR2869069} and \cite{MR3118409}. 

Finally, the \emph{bungee} set is defined simply as $BU(f) := \C \setminus (I(f) \cup BO(f))$. It is easy to see that if $P$ is a polynomial, then there is a punctured neighbourhood of infinity that lies in $I(P)$, and so $BU(P)$ is empty. However, if $f$ is transcendental, then $BU(f)$ is non-empty. In fact, see \cite[Theorem 5.1]{OsSix}, the Hausdorff dimension of $BU(f) \cap J(f)$ is greater than zero. The properties of $BU(f)$ were studied in \cite{Lazebnik2017611} and in \cite{OsSix}. 

From here onwards we assume that $f$ is transcendental. It is now known that the set $\Ihat(f) := I(f) \cup \{\infty\}$ is a connected subset of the Riemann sphere; see \cite[Theorem 4.1]{MR2801622}. Moreover, see \cite[Theorem 1.1]{Kfc}, the same property holds for $I(f) \cup BU(f) \cup \{\infty\}$. Our principle interest in this paper is to ask if there are conditions that ensure that one or both of the sets 
\[
\BOhat(f) := BO(f) \cup \{\infty\} \quad\text{ and }\quad \BUhat(f) := BU(f) \cup \{\infty\}
\]
can also be connected. In fact we have the following result.
\begin{theorem}
\label{theo:example}
There is a {\tef} $f$ such that each of the sets $\Ihat(f)$, $\BOhat(f)$ and $\BUhat(f)$ is connected.
\end{theorem}

\begin{remark}\normalfont
The function $f$ in Theorem~\ref{theo:example} has only two singular values (points at which it is not possible to define some inverse branch). It follows that $f$ is in the \emph{Speiser} class $\mathcal{S}$, which consists of those {\tef}s for which the set of singular values is finite.
\end{remark}

Since $\Ihat(f)$ is always connected, in order to prove Theorem~\ref{theo:example} we give sufficient conditions for $\BOhat(f)$ or $\BUhat(f)$ to be connected. We then show that there is a {\tef} with the necessary properties. The first result is as follows.

\begin{theorem}
\label{theo:BOhat}
Suppose that $f$ is a {\tef}. If $f$ has an unbounded Fatou component in $BO(f)$ (resp. $BU(f)$), then $\BOhat(f)$ (resp. $\BUhat(f)$) is connected.
\end{theorem}
\begin{remark}\normalfont
The usefulness of the second part of this theorem is limited by the fact that there are relatively few examples of {\tef}s with Fatou components in $BU(f)$. Examples of a {\tef} with such a Fatou component were given in \cite{Bish3} and \cite{MR918638}. The only examples of Fatou components in $BU(f)$ which are also known to be unbounded were given in \cite{Lazebnik2017611} and \cite{OsSix}.
\end{remark}

The second result requires the notion of a finite logarithmic asymptotic value, which we define as follows. We use the notation 
\[
B(a, r) := \{ z \in \C : |z - a| < r \}, \qfor a \in \C \text{ and } r > 0. 
\]
\begin{definition}
Suppose that $f$ is a {\tef}. A value $\alpha\in\C$ is a \emph{finite logarithmic asymptotic value} of $f$ if there exist $r>0$ and a component $U$ of $f^{-1}(B(\alpha,r))$, such that the restriction $f :U \to B(\alpha, r) \setminus\{\alpha\}$ is a universal covering.
\end{definition}

Our sufficient condition for the connectedness of $\BUhat(f)$ is as follows.

\begin{theorem}
\label{theo:BUhat}
Suppose that $f$ is a {\tef}. If $f$ has a finite logarithmic asymptotic value $\alpha \in J(f)$, then $\BUhat(f)$ is connected.
\end{theorem}

\begin{remark}\normalfont
In \cite{OsSix} the authors asked if there is a {\tef} $f$ such that $BU(f)$ is connected. Although this question is still open, these results give at least a partial answer to this question.
\end{remark}

Recent study of $I(f)$ has shown that this set often has a topological structure known as a {\spw}. The following definition of a {\spw} was first given in \cite{Rippon01102012}.
\begin{definition}
A connected set $E\subset\C$ is a \emph{\spw} if there exists a sequence of bounded simply connected domains, $(G_n)_{n\in\N}$, such that 
\[
\partial G_n \subset E, \ G_n \subset G_{n+1}, \text{ for } n\in\N, \text{ and } \bigcup_{n\in\N} G_n = \C.
\]
\end{definition}
Clearly if $I(f)$ is a {\spw}, then neither $\BOhat(f)$ nor $\BUhat(f)$ can be connected. In fact, in some sense, the converse is also true; if $\BOhat(f) \cup \BUhat(f)$ is disconnected, then $I(f)$ is a {\spw}. See Corollary~\ref{cor:conn} below.

There are now many examples of {\tef}s $f$ such that $I(f)$ is a {\spw}; see, for example, \cite{Vasso1}, \cite{Rippon01102012} and \cite{MR2838342}. However, none of these examples are in the much studied \emph{Eremenko-Lyubich} class $\B$; this class consists of those {\tef}s for which the set of singular values is bounded. The techniques used to prove our earlier results can be used to show that there is a large subclass of class $\B$ for which the escaping set is not a {\spw}.
\begin{theorem}
\label{theo:classb}
Suppose that $f \in \B$ is a {\tef}. If $f$ has a finite logarithmic asymptotic value, then $I(f)$ is not a {\spw}.
\end{theorem}
In fact we conjecture the following.
\begin{conjecture}
If $f \in \B$ is a {\tef}, then $I(f)$ is not a {\spw}. 
\end{conjecture}

Our final result is the following, which gives a simple topological characterisation of an $I(f)$ {\spw}, and also an $A(f)$ {\spw}, for a {\tef} $f$. Here $A(f)$ is the so-called \emph{fast escaping set}, which was introduced in \cite{MR1684251}, and can be defined, see \cite{Rippon01102012}, by;
\begin{equation}
\label{Adef}
A(f) := \{z \in \C : \text{there exists } \ell \in\N \text{ such that } |f^{n+\ell}(z)| \geq M^n(R,f), \text{ for } n \in\N \}.
\end{equation}
Here the \emph{maximum modulus function} is defined by $M(r,f) := \max_{|z|=r} |f(z)|,$ for $r \geq 0.$ We write $M^n(r,f)$ to denote repeated iteration of $M(r,f)$ with respect to the variable $r$. In \eqref{Adef}, we assume that $R > 0$ is sufficiently large that $M^n(R,f)\rightarrow\infty$ as $n\rightarrow\infty$. Finally, we say that a set $E \subset \C$ \emph{separates a point $z \in \C$ from infinity} if there is a bounded open set $U$ such that $z \in U$ and $\partial U \subset E$.
\begin{theorem}
\label{theo:spwcriterion}
Suppose that $f$ is a {\tef}. Then $I(f)$ (resp. $A(f)$) is a {\spw} if and only if it separates some point of $J(f)$ from infinity.  If $f$ has no multiply connected Fatou components, then $J(f)$ is a {\spw} if and only if it separates some point of $J(f)$ from infinity.
\end{theorem}
\begin{remark}\normalfont
It is known that if $f$ is a {\tef}, then $I(f)$ contains an unbounded component; see \cite[Theorem 1.1]{Rippon01102012}. This implies that neither $BO(f)$ nor $BU(f)$ can be a {\spw}.
\end{remark}
\subsection*{Structure of the paper}
The structure of this paper is as follows. First, in Section~\ref{S:prelim} we gather some preliminary results. Next, in Section~\ref{S:nonB} we prove Theorem~\ref{theo:BOhat} and Theorem~\ref{theo:BUhat}, and then use these results to prove Theorem~\ref{theo:example}. Finally, in Section~\ref{S:B} we prove Theorem~\ref{theo:classb} and Theorem~\ref{theo:spwcriterion}.

%
%
%
\section{Preliminary results}
\label{S:prelim}
We use the following, which is known as the ``blowing-up'' property of the Julia set; see, for example, \cite[Lemma 2.2]{MR1216719}. Here an exceptional point is a point with finite backward orbit; there is at most one such point.
\begin{lemma}
\label{lemm:blowup}
Suppose that $f$ is a {\tef}, and $V$ is an open set that meets $J(f)$. If $K$ is a compact set that does not contain an exceptional point, then there exists $n_0 \in \N$ such that $f^n(V) \supset K$, for $n \geq n_0$.
\end{lemma}
We also require a result on wandering domains. If $f$ is a {\tef}, and $U$ is a Fatou component of $f$, then we say that $U$ is \emph{preperiodic} if there exist $n, m \in \N$ with $n \ne m$ and $f^n(U) = f^m(U)$. If this is not the case, then we say that $U$ is \emph{wandering}. We use the following \cite[Theorem 1.5]{OsSix} which, roughly speaking, says that most points on the boundary of a wandering domain have the same behaviour under iteration as the domain itself. Here, for a \tef\ $f$, the \emph{$\omega$-limit set} $\omega(z,f)$ of a point $z\in\C$ is the set of accumulation points of its orbit in $\widehat{\C}$.  For a wandering domain $U$ of $f$, it follows by normality that $\omega(z_1,f) = \omega(z_2,f)$ for $z_1, z_2 \in U$, so in this case we can write $\omega(U,f)$ without ambiguity.
\begin{lemma}
\label{lemm:wandering}
Suppose that $f$ is a {\tef} and that $U$ is a wandering domain of $f$. Then the set $\{ z \in \partial U : \omega(z,f) \ne \omega(U, f) \}$
has harmonic measure zero relative to $U$.
\end{lemma}
We need the following, which is \cite[Lemma 3.1]{Sixsmithmax}.
\begin{lemma}
\label{lemm:RS}
Suppose that $(E_n)_{n\in\N}$ is a sequence of compact sets and $(m_n)_{n\in\N}$ is a sequence of integers. Suppose also that $f$ is a {\tef} such that $E_{n+1} \subset f^{m_n}(E_n )$, for $n\in\N$. Set $p_n = \sum_{k=1}^n m_k$, for $n\in\N$. Then there exists $\zeta\in E_1$ such that 
\begin{equation}
\label{feq}
f^{p_n}(\zeta) \in E_{n+1}, \qfor n\in\N.
\end{equation}
If, in addition, $E_n \cap J(f) \ne \emptyset$, for $n\in\N$, then there exists $\zeta \in E_1 \cap J(f)$ such that (\ref{feq}) holds.
\end{lemma}
To prove Theorem~\ref{theo:BUhat} we require the following, which seems to be new.
\begin{lemma}
\label{lemm:tract}
Let $f$ be a {\tef} with a finite logarithmic asymptotic value $\alpha\in\C$. Let $r>0$ be sufficiently small that $f$ is a universal covering from a component $T$ of $f^{-1}(B(\alpha, r))$ to $B(\alpha, r)\setminus\{\alpha\}$. Then there exist $R>0$, and a component $V$ of $T \cap B(0, R)$ such that the following holds. Suppose that $\gamma \subset \C\setminus \overline{B(0, R)}$ is a continuum such that $V$ lies in a bounded component of $T\setminus \gamma$. Then the complementary component of $f(\gamma \cap T)$ containing $\alpha$ lies in $B(\alpha,r)$.
\end{lemma}
\begin{proof}

Let $r$ and $T$ be as in the statement of the lemma. Choose $R>0$ sufficiently large that there is a component $V$ of $T \cap B(0, R)$, such that $f(V)$ contains an annulus of the form $A := \{ z \in \C : r-\delta < |z-\alpha| < r \}$, for some $\delta \in (0, r)$.

Now, suppose that $\gamma \subset \C\setminus \overline{B(0, R)}$ is a continuum such that $V$ lies in a bounded component of $T\setminus \gamma$. Let $S$ be the component of $\C \setminus f(\gamma\cap T)$ that contains $\alpha$. We need to show that $S \subset B(\alpha, r)$. Suppose, therefore, that this is not the case. Then $S$ contains both $\alpha$ and a point $\zeta \in A$. Since $S$ is a domain, we can let $\Gamma$ be a curve in $S$ that joins $\alpha$ and $\zeta$. Without loss of generality (replacing $\zeta$ with some other point of $\Gamma\cap A$ if necessary), we can assume that $\Gamma \subset B(\alpha, r)$.

Let $\zeta' \in V$ be a preimage of $\zeta$, and let $\Gamma'$ be the component of $f^{-1}(\Gamma)$ that contains $\zeta'$. Then $\Gamma' \subset T$ joins a point in $V$ to a point in an unbounded component of $T \setminus \gamma$, which is a contradiction.
\end{proof}
%
%
%
\section{Results on $\Ihat(f), \BOhat(f)$, and $\BUhat(f)$}
\label{S:nonB}
Theorem~\ref{theo:BOhat} is a consequence of the following lemma.
\begin{lemma}
\label{lemm:BOhat}
Suppose that $f$ is a {\tef} with an unbounded Fatou component in $BO(f)$ (resp. $BU(f)$). Suppose that $U$ is a bounded domain that meets $BO(f)$ (resp. $BU(f)$). Then $\partial U$ also meets $BO(f)$ (resp. $BU(f)$).
\end{lemma}
\begin{proof}
We prove only the case of $BO(f)$. The case of $BU(f)$ is very similar, and is omitted.

Suppose first that $U \subset F(f)$. If $U$ is not itself a Fatou component, then the result follows by normality. Hence we can assume that $U$ is a Fatou component of $f$.

If $U$ is preperiodic, then it is easy to see that $\partial U \subset BO(f)$. Hence we can assume that $U$ is wandering. The conclusion of the lemma then follows by Lemma~\ref{lemm:wandering}.

We can assume, therefore, that $U$ meets $J(f)$. Let $V$ be the unbounded Fatou component in $BO(f)$. It follows by Lemma~\ref{lemm:blowup} that there exists $n \in \N$ such that $f^n(U) \cap V \ne \emptyset$. Hence $f^n(\partial U) \cap V \ne \emptyset$, and the result follows.
\end{proof}
\begin{proof}[Proof of Theorem~\ref{theo:BOhat}]
As in the case of Lemma~\ref{lemm:BOhat} we prove only the case of $BO(f)$. The case of $BU(f)$, which is very similar, is omitted.

Suppose that, with the hypotheses of the theorem, $\BOhat(f)$ was not connected. Then there would be disjoint open sets $H_1, H_2 \subset \widehat{\C}$ such that
\[
\BOhat(f) \subset H_1 \cup H_2 \quad\text{ and }\quad \partial H_i \cap BO(f) \ne \emptyset, \text{ for } i \in \{1, 2\}.
\]
Without loss of generality we can assume that $H_1$ is bounded and meets $BO(f)$. It follows by Lemma~\ref{lemm:BOhat} that $\partial H_1$ meets $BO(f)$, which is a contradiction, completing the proof.
\end{proof}
Theorem~\ref{theo:BUhat} is a consequence of the following, which clearly is analogous to Lemma~\ref{lemm:BOhat}.
\begin{lemma}
\label{lemm:BUhat}
Suppose that $f$ is a {\tef}, and that $f$ has a finite logarithmic asymptotic value in $J(f)$. Suppose that $U$ is a bounded domain that meets $BU(f)$. Then $\partial U$ also meets $BU(f)$.
\end{lemma}
\begin{proof}
Suppose first that $U \subset F(f)$. If $U$ is not itself a Fatou component, then the result follows by normality. Hence we can assume that $U$ is a Fatou component of $f$. It is known \cite[Theorem 1.1]{OsSix} that $U$ must be wandering. The result then follows by Lemma~\ref{lemm:wandering}.

We can assume, therefore, that $U$ meets $J(f)$. Let $\alpha \in J(f)$ be a finite logarithmic asymptotic value of $f$. Let $r, R > 0$, let $T$ be a component of $f^{-1}(B(\alpha, r))$, and let $V$ be the component of $T \cap B(0, R)$, such that the properties stated in Lemma~\ref{lemm:tract} all hold. Let $(R_n)_{n\in\N}$ be a sequence of real numbers larger than $R$ that tend to infinity. Let $W$ be a bounded open disc containing any exceptional point of $f$. We can assume that $\overline{W} \subset B(0, R_n)$, for $n \in \N$.

We now construct a point in $\partial U \cap BU(f)$. Set $U_1 := U$. By Lemma~\ref{lemm:blowup}, there exists $n_1 \in \N$ such that $\overline{B(0,R_1)}\setminus W \subset f^{n_1}(U_1)$. Hence there is a continuum $E_1 \subset \partial U_1$ such that $V$ lies in a bounded component of $T \setminus f^{n_1}(E_1)$.

Now consider $f(f^{n_1}(E_1) \cap T)$. It follows by an application of Lemma~\ref{lemm:tract} that the complement of $f(f^{n_1}(E_1) \cap T)$ has a simply connected component containing $\alpha$, and lying in $B(\alpha,r)$. Call this component $U_2$. Note that $\partial U_2 \subset f^{n_1 + 1}(E_1)$.

Since $U_2$ meets $J(f)$ -- recall that $\alpha \in J(f)$ -- we can iterate the above construction. We obtain a sequence of integers $(n_k)_{k\in\N}$ and a sequence of continua $(E_k)_{k\in\N}$ such that 
\begin{equation}
\label{eq:theE}
E_{k+1} \subset f^{n_k+1}(E_k), \ f^{n_k}(E_k) \subset \C\setminus \overline{B(0, R_k)}, \text{ and } f^{n_k+1}(E_k) \subset \overline{B(\alpha, r)}, \qfor k \in \N.
\end{equation}

It follows by Lemma~\ref{lemm:RS} that there is a point $\zeta\in E_1$ and a sequence of integers $(p_k)_{k\in\N}$ such that $f^{p_k}(\zeta) \in E_{k+1}$, for $k\in\N$. It follows from \eqref{eq:theE} that $\zeta \in \partial U \cap BU(f)$, as required.
\end{proof} 
Finally in this section, we use these results to prove Theorem~\ref{theo:example}.
\begin{proof}[Proof of Theorem~\ref{theo:example}]
We construct a {\tef} $f$ with a finite logarithmic asymptotic value in $J(f)$, and a second finite logarithmic asymptotic value lying in $F(f) \cap BO(f)$. It is easy to deduce from the second fact that $f$ has an unbounded Fatou component in $BO(f)$. The result then follows by Theorem~\ref{theo:BOhat} and Theorem~\ref{theo:BUhat}.

Consider the family of functions
\[
                    f_{\alpha, \beta}(z) := \frac{2\alpha}{\sqrt{\pi}} \int_0^z e^{-w^2} \ dw + \beta = \alpha \operatorname{erf}(z) + \beta, \qfor \alpha, \beta \in \C, \ \alpha \ne 0.
\]
Here $\operatorname{erf}(z)$ denotes the error function; see \cite[p.297]{standards}.

Clearly $f_{\alpha, \beta}$ has no critical values. It can be seen that $f_{\alpha, \beta}$ has two finite logarithmic asymptotic values, obtained as $z$ tends to infinity along the real axis in the positive and negative directions. It is a calculation to show that these asymptotic values are equal to $\pm \alpha + \beta$.

We choose values for $\alpha$ and $\beta$ so that $\alpha + \beta \in J(f_{\alpha, \beta})$ and $-\alpha + \beta$ lies in a parabolic basin of $f_{\alpha, \beta}$. First we let $c$ be a complex solution to $\operatorname{erf}(z) = 1$. In particular, we set $c = -1.3548101281\ldots + 1.9914668428\ldots i$; see \cite[Table 7.13.2]{NIST:DLMF}. We then let
\[
\alpha = \frac{e^{c^2}\sqrt{\pi}}{2} \quad\text{and}\quad\beta=c-\alpha.
\]
It follows that 
\[
f_{\alpha, \beta}(\alpha + \beta) = f_{\alpha, \beta}(c) = \alpha + \beta, \quad\text{and}\quad f'(\alpha+\beta) = \frac{2\alpha}{\sqrt{\pi}} e^{-c^2} = 1.
\]
Hence $\alpha + \beta$ is a parabolic fixed point and so lies in the Julia set of $f_{\alpha, \beta}$. It can be seen from Figure~\ref{fig:2} that $-\alpha+\beta$ lies in the parabolic basin of this point. This is exactly what we require.
\end{proof}

 \begin{figure}
  \subfloat{\includegraphics[width=1\textwidth]{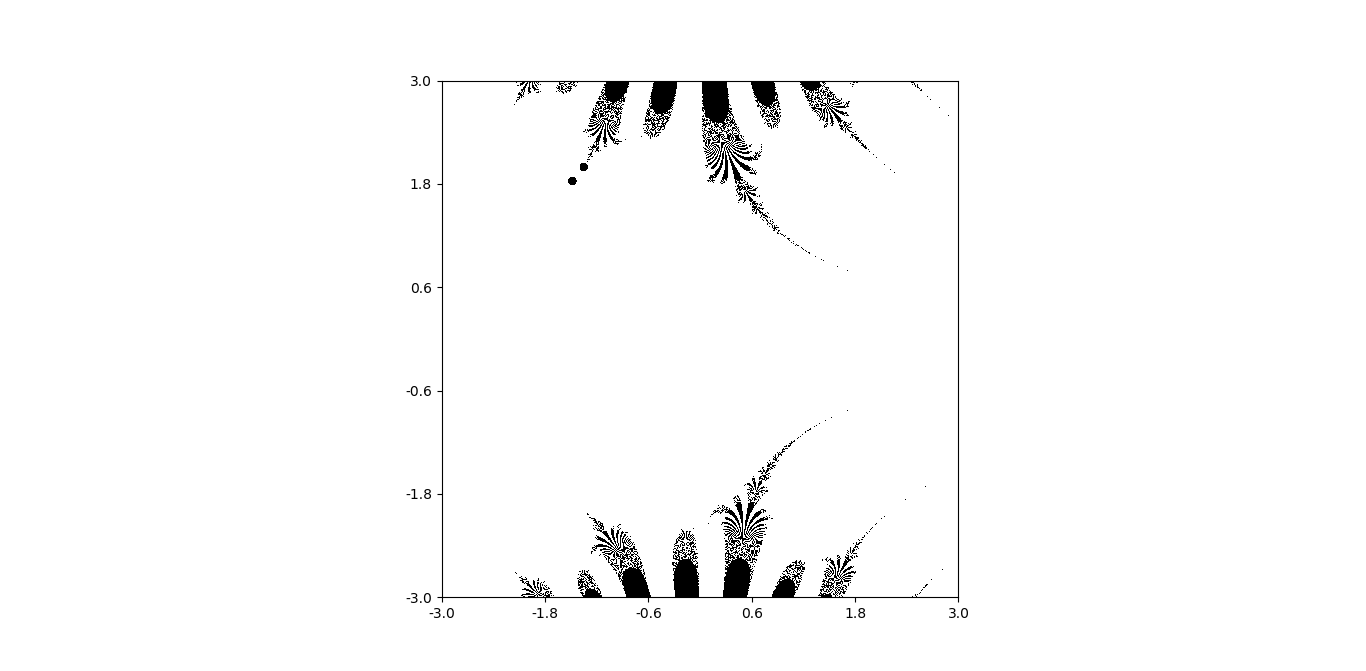}}\hfill
  \subfloat{\includegraphics[width=1\textwidth]{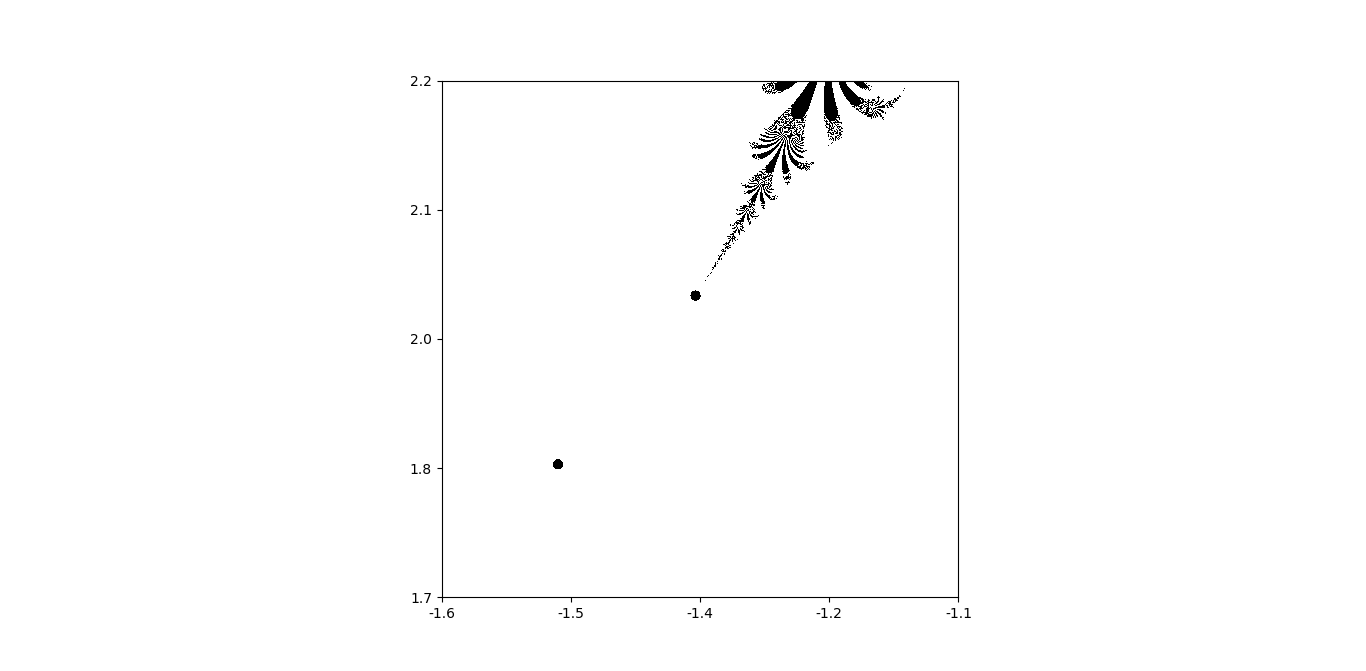}}
  \caption{Two views of the Julia set (black) of $f_{\alpha, \beta}$. The asymptotic values $\pm\alpha+\beta$ are denoted by black circles; in particular $\alpha+\beta$ is the top right circle, and is a parabolic fixed point.}\label{fig:2}
 \end{figure}

\begin{remarks}\normalfont
\mbox{ }
\begin{enumerate}
\item It follows by \cite[Theorem 1.2]{MR2753600} that every component of $I(f_{\alpha, \beta})$ is unbounded and path-connected; in the terminology of \cite{AnnaLasse}, $f_{\alpha, \beta}$ is \emph{criniferous}. It can be seen also that this function is \emph{postsingularly bounded}. We refer to \cite{AnnaLasse} for a definition, and further information on functions with this property.
\item From the computer pictures, it appears that $F(f_{\alpha, \beta})$ is connected, although we have not been able to prove this. If this were indeed the case, then it would follow at once that $BO(f)$ and $\BOhat(f)$ are connected.
\item Clearly, another approach to the proof of Theorem~\ref{theo:example} would be to find a {\tef} $f$ with an unbounded Fatou component in $BO(f)$ and an unbounded Fatou component in $BU(f)$; the result would follow by Theorem~\ref{theo:BOhat}. While this seems possible, it is also likely to be more complicated than the example given here, because of the difficulty of constructing {\tef}s with an unbounded Fatou component in $BU(f)$.
\end{enumerate}
\end{remarks}


%
%
%
\section{Results on spiders' webs}
\label{S:B}
%
Recall that if $f \in \B$, then $I(f) \subset J(f)$; see \cite[Theorem 1]{MR1196102}. It follows that Theorem~\ref{theo:classb} is an immediate consequence of the following.
\begin{theorem}
\label{theo:noloop}
Suppose that $f$ is a {\tef} with a finite logarithmic asymptotic value $\alpha \notin F(f) \cap I(f)$. Then $I(f)$ separates no finite point from infinity.
\end{theorem}

\begin{proof}
Suppose first that $\alpha \in J(f)$. It follows that $\BUhat(f)$ is connected, by Theorem~\ref{theo:BUhat}. Hence $I(f)$ is not a {\spw}. On the other hand, if $\alpha \in F(f)$, then either $\alpha \in BO(f)$ or $\alpha \in BU(f)$. In these cases the result follows by Theorem~\ref{theo:BOhat}.
\end{proof}

Next we prove Theorem~\ref{theo:spwcriterion}.
\begin{proof}[Proof of Theorem~\ref{theo:spwcriterion}]
One direction is immediate; it is easy to see from the definition that a set separates every point of $\C$ from infinity if it is a {\spw}.

In the other direction, we consider first the case for $I(f)$. Suppose that $I(f)$ separates a point of $J(f)$ from infinity. In other words, there is a bounded open set $U$ that meets $J(f)$ and the boundary of which lies in $I(f)$. Let $W$ be a disc containing any exceptional point of $f$. Suppose that $R>0$ is sufficiently large that $W \subset B(0, R)$. By Lemma~\ref{lemm:blowup} there exists $n = n(R) \in \N$ such that $\overline{B(0, R)}\setminus W \subset f^n(U)$. Since $I(f)$ is forward invariant, it follows that for all sufficiently large $R>0$, there is a bounded simply connected domain $G$ such that $B(0, R) \subset G$ and $\partial G \subset I(f)$.

The fact \cite[Theorem 1.1]{Rippon01102012} that $I(f)$ contains an unbounded component implies that $I(f)$ contains a {\spw}. The remark \cite[p.807]{Rippon01102012} then implies that $I(f)$ is a {\spw}.

The case for $A(f)$ is almost identical to that of $I(f)$, and is omitted. If $f$ has no multiply connected Fatou components, then it follows from \cite[Theorem 1]{MR1609471} that all components of $J(f)$ are unbounded. The result for $J(f)$ then follows in much the same way as that of $I(f)$.
\end{proof}
The following, which was promised in the introduction, is now quite straightforward.
\begin{corollary}
\label{cor:conn}
Suppose that $f$ is a {\tef}. Then $I(f)$ is a {\spw} if and only if $\BOhat(f) \cup \BUhat(f)$ is disconnected.
\end{corollary}
\begin{proof}
One direction is immediate, and so we assume that $\BOhat(f) \cup \BUhat(f)$ is disconnected. Then there is a bounded open set $U$ such that $U$ meets $BO(f) \cup BU(f)$ and $\partial U \subset I(f)$. Arguing as in the proof of Lemma~\ref{lemm:BOhat}, we can deduce that $U \cap J(f) \ne \emptyset$. The result now follows by Theorem~\ref{theo:spwcriterion}.
\end{proof}
%
%
%
%
%
\emph{Acknowledgment:} The author is grateful to Lasse Rempe-Gillen, Phil Rippon, Gwyneth Stallard and John Osborne for useful conversations, and also to Ben and Pete Strulo for their help with Figure~\ref{fig:2}. 
\bibliographystyle{alpha}
\bibliography{../../../Research.References}

\begin{thebibliography}{RRRS11}

\bibitem[AS72]{standards}
Milton Abramowitz and Irene~A. Stegun, editors.
\newblock {\em Handbook of Mathematical Functions with Formulas, Graphs, and
  Mathematical Tables}.
\newblock National Bureau of Standards, 1972.

\bibitem[Ben17]{AnnaLasse}
L.~Benini, A; Rempe-Gillen.
\newblock A landing theorem for entire functions with bounded post-singular
  sets.
\newblock {\em Preprint, arXiv:1711.10780v2}, 2017.

\bibitem[Ber93]{MR1216719}
Walter Bergweiler.
\newblock Iteration of meromorphic functions.
\newblock {\em Bull. Amer. Math. Soc. (N.S.)}, 29(2):151--188, 1993.

\bibitem[Ber12]{MR2869069}
Walter Bergweiler.
\newblock On the set where the iterates of an entire function are bounded.
\newblock {\em Proc. Amer. Math. Soc.}, 140(3):847--853, 2012.

\bibitem[BH99]{MR1684251}
W.~Bergweiler and A.~Hinkkanen.
\newblock On semiconjugation of entire functions.
\newblock {\em Math. Proc. Cambridge Philos. Soc.}, 126(3):565--574, 1999.

\bibitem[Bis15]{Bish3}
Christopher~J. Bishop.
\newblock Constructing entire functions by quasiconformal folding.
\newblock {\em Acta Mathematica}, 214(1):1--60, 2015.

\bibitem[{\relax DLMF}]{NIST:DLMF}
{NIST} {D}igital {L}ibrary of {M}athematical {F}unctions.
\newblock http://dlmf.nist.gov/, Release 1.0.16 of 2017-09-18.
\newblock F.~W.~J. Olver, A.~B. {Olde Daalhuis}, D.~W. Lozier, B.~I. Schneider,
  R.~F. Boisvert, C.~W. Clark, B.~R. Miller and B.~V. Saunders, eds.

\bibitem[EL87]{MR918638}
A.~E. Eremenko and M.~Yu. Lyubich.
\newblock Examples of entire functions with pathological dynamics.
\newblock {\em J. Lond. Math. Soc. (2)}, 36(3):458--468, 1987.

\bibitem[EL92]{MR1196102}
A.~E. Eremenko and M.~Yu. Lyubich.
\newblock Dynamical properties of some classes of entire functions.
\newblock {\em Ann. Inst. Fourier (Grenoble)}, 42(4):989--1020, 1992.

\bibitem[Ere89]{MR1102727}
A.~E. Eremenko.
\newblock On the iteration of entire functions.
\newblock {\em Dynamical systems and ergodic theory ({W}arsaw 1986)},
  23:339--345, 1989.

\bibitem[Evd16]{Vasso1}
V.~Evdoridou.
\newblock Fatou's web.
\newblock {\em Proc. Amer. Math. Soc.}, 144(12):5227--5240, 2016.

\bibitem[Kis98]{MR1609471}
Masashi Kisaka.
\newblock On the connectivity of {J}ulia sets of transcendental entire
  functions.
\newblock {\em Ergodic Theory Dynam. Systems}, 18(1):189--205, 1998.

\bibitem[Laz17]{Lazebnik2017611}
Kirill Lazebnik.
\newblock Several constructions in the {E}remenko-{L}yubich class.
\newblock {\em Journal of Mathematical Analysis and Applications}, 448(1):611
  -- 632, 2017.

\bibitem[ORS17]{Kfc}
J.~W. Osborne, P.~J. Rippon, and G.~M. Stallard.
\newblock Connectedness properties of the set where the iterates of an entire
  function are unbounded.
\newblock {\em Ergodic Theory Dynam. Systems}, 37(4):1291--1307, 2017.

\bibitem[OS16]{OsSix}
John~W. Osborne and David~J. Sixsmith.
\newblock On the set where the iterates of an entire function are neither
  escaping nor bounded.
\newblock {\em Ann. Acad. Sci. Fenn. Math.}, 41(2):561--578, 2016.

\bibitem[Osb13]{MR3118409}
John~W. Osborne.
\newblock Connectedness properties of the set where the iterates of an entire
  function are bounded.
\newblock {\em Math. Proc. Cambridge Philos. Soc.}, 155(3):391--410, 2013.

\bibitem[RRRS11]{MR2753600}
G{\"u}nter Rottenfusser, Johannes R{\"u}ckert, Lasse Rempe, and Dierk
  Schleicher.
\newblock Dynamic rays of bounded-type entire functions.
\newblock {\em Ann. of Math. (2)}, 173(1):77--125, 2011.

\bibitem[RS11]{MR2801622}
P.~J. Rippon and G.~M. Stallard.
\newblock Boundaries of escaping {F}atou components.
\newblock {\em Proc. Amer. Math. Soc.}, 139(8):2807--2820, 2011.

\bibitem[RS12]{Rippon01102012}
P.~J. Rippon and G.~M. Stallard.
\newblock Fast escaping points of entire functions.
\newblock {\em Proc. London Math. Soc. (3)}, 105(4):787--820, 2012.

\bibitem[Six11]{MR2838342}
D.~J. Sixsmith.
\newblock Entire functions for which the escaping set is a spider's web.
\newblock {\em Math. Proc. Cambridge Philos. Soc.}, 151(3):551--571, 2011.

\bibitem[Six15]{Sixsmithmax}
David~J. Sixsmith.
\newblock Maximally and non-maximally fast escaping points of transcendental
  entire functions.
\newblock {\em Math. Proc. Cambridge Philos. Soc.}, 158(2):365--383, 2015.

\end{thebibliography}
\end{document}